\documentclass[a4paper, 10pt, twoside, notitlepage]{amsart}

\usepackage[utf8]{inputenc}
\usepackage{color}
\usepackage{amsmath} 
\usepackage{amssymb} 
\usepackage{amsthm}
\usepackage{graphicx}
\usepackage{esint}
\usepackage{subcaption}
\usepackage[colorlinks=true,linkcolor=blue]{hyperref}

\theoremstyle{plain}
\newtheorem{thm}{Theorem}
\newtheorem{prop}{Proposition}[section]
\newtheorem{lem}[prop]{Lemma}

\newtheorem{rmk}[prop]{Remark}

\newtheorem{example}[prop]{Example}

\newcommand {\R} {\mathbb{R}} \newcommand {\Z} {\mathbb{Z}}
 \newcommand {\N} {\mathbb{N}}

\newcommand {\diam} {\text{diam}}

\DeclareMathOperator{\argmin}{argmin}

\DeclareMathOperator {\dist} {dist}

\DeclareMathOperator{\Curl} {Curl}
\DeclareMathOperator{\Per} {Per}

\newcommand{\norm}[1]{\left|\left|#1\right|\right|}
\newcommand{\abs}[1]{\left|#1\right|}
\newcommand{\rB}[1]{\ensuremath{\left(#1\right)}}

\newcommand{\cB}[1]{\ensuremath{\left\{#1\right\}}}
\DeclareMathOperator{\de}{d}
\DeclareMathOperator{\Dim}{dim}

\usepackage{xspace}

\title[A Compactness and Structure Result]{A Compactness and Structure Result for a Discrete
Multi-Well Problem with $SO(n)$ Symmetry in Arbitrary Dimension}

\author[G. Kitavtsev]{Georgy Kitavtsev}
\address{
University of Bristol,
University Walk,
Clifton, Bristol BS8 1TW, United Kingdom}
\email{georgy.kitavtsev@bristol.ac.uk}

\author[G. Lauteri]{Gianluca Lauteri}
\address{
Max-Planck-Institute for Mathematics in the Sciences,
Inselstraße 22,
04103 Leipzig,
Germany}
\email{lauteri@mis.mpg.de}

\author[S. Luckhaus]{Stephan Luckhaus}
\address{Universität Leipzig,
Mathematisches Institut,
D-04009 Leipzig, Germany}
\email{luckhaus@math.uni-leipzig.de}

\author[A. R\"uland]{Angkana R\"uland}
\address{
Max-Planck-Institute for Mathematics in the Sciences,
Inselstraße 22,
04103 Leipzig,
Germany}
\email{rueland@mis.mpg.de}

\begin{document}

\begin{abstract}
In this note we combine the ``spin-argument" from \cite{KLR15} and the
$n$-dimensional incompatible, one-well rigidity result from \cite{LL16}, in order to
infer a new proof for the compactness of discrete multi-well energies associated
with the modelling of surface energies in certain phase transitions. Mathematically, a main novelty here
is the reduction of the problem to an incompatible one-well problem. The presented argument is very robust and applies to a number of different physically interesting models, including for instance phase transformations in shape-memory materials but also anti-ferromagnetic transformations or related transitions with an ``internal" microstructure on smaller scales.
\end{abstract}

\maketitle

\section{Introduction}
It is the purpose of this note to provide a short, essentially self-contained
compactness argument and structure result for a multi-well discrete-to-continuum problem with $SO(n)$ symmetry arising in the
variational modelling of certain phase transitions in arbitrary dimension by combining the
ingredients from \cite{KLR15} and \cite{LL16}, \cite{LL17}. Although similar results were already proved in the context of martensitic phase transitions occurring in the modelling of shape-memory materials in \cite{KLR15} and \cite{ALP17}, the present article contains two main novelties: On the one hand, the mathematical argument leading to the structure result is of interest, since in the present note we neither invoke the
rigidity results of \cite{DM1} nor of \cite{CGP07}. Instead, we reduce the problem to the setting of an (incompatible) one-well problem.
On the other hand, our arguments extend to a large class of Hamiltonians with possibly quite nonlocal interactions. These cover a number of relevant physical phase transformations, including for instance also the anti-ferromagnetic transitions. The generality of the systems which are covered here, goes far beyond the ones from \cite{KLR15} or \cite{ALP17}.\\

From a mathematical point of view, the result is based on the following idea: Instead of relying on \cite{DM1} or on \cite{CGP07}, we combine the ``spin
argument" from \cite{KLR15} with the $n$-dimensional strong-supercritical and weak-critical, incompatible, one-well rigidity results from \cite{LL17} and their consequences from \cite{LL16}. To this end, we reduce the compatible
multi-well problem to an auxiliary, incompatible one-well setting, which allows us
to invoke \cite{LL16}. We believe that this
strategy -- and in particular the reduction to the auxiliary one-well problem -- is
interesting in its own right and can be applied to a quite wide range of physically relevant models.
\\

\subsection{Outline of the article}
In order to introduce the new ideas in an as simple as possible set-up, the first part of the article (Sections \ref{sec:setup_result}-\ref{sec:proof}) deals with a special Hamiltonian and only considers the physical setting of a martensitic transition in shape-memory alloys. For a nearest neighbour Hamiltonian, which is easy to formulate and which is presented in Section \ref{sec:setup_result}, we explain the interplay between the spin argument (c.f. Section \ref{sec:spin}) and the incompatible, one-well rigidity results from \cite{LL17} (c.f. Section \ref{sec:proof}).   
As the main result in this context, we derive the structure result formulated in Theorem \ref{prop:BV_energy}.

In Section \ref{sec:gen}, we then generalize these ideas to a much larger class of physical systems and a much larger family of underlying Hamiltonians. Here we for instance allow for arbitrary finite range interactions. Also in this setting we deduce a compactness and structure result (see Theorem \ref{thm:BV_struc_gen}).
This in particular shows the robustness of the underlying mathematical argument.

\section{The Set-Up and the Main Result for Martensitic Phase Transformations Modelled by Nearest Neighbour Interactions}
\label{sec:setup_result}

\subsection{Set-up}
\label{sec:setup}

In order to formulate our main result in the context of nearest neighbour interaction models for martensitic phase transformations, we first outline the precise set-up of our
$n$-dimensional, discrete multi-well problem. Here we impose several conditions on
the wells, the underlying triangulation and the associated energy, which we explain
in the sequel. In order to keep the set-up as simple as possible, we first discuss a model scenario and postpone the analysis of more general systems (including more nonlocal interactions and periodic internal microstructures which may be present in other physical systems) to Section \ref{sec:gen}.\\

We begin by introducing the relevant ingredients in formalizing the precise setting.

\subsubsection{Wells}
\label{sec:wells}

Let $U_1,\dots,U_k \in \R^{n\times n}$ be pairwise different, symmetric, positive
definite matrices.
Let $K= \bigcup\limits_{j=1}^{k}SO(n)U_j$ denote the union of the \emph{(energy)
wells} $SO(n)U_j$. Assume that the wells are pairwise rank-one connected, i.e.
assume that there exist two rotations $Q_{ij}^{\pm}$ and vectors $a_{ij}^{\pm} \in
\R^n \setminus \{0\}$, $b_{ij}^{\pm}\in S^{n-1}$ such that
\begin{align}
\label{eq:rank_one}
U_i - Q_{ij}^{\pm} U_j = a_{ij}^{\pm}\otimes b_{ij}^{\pm}.
\end{align}
Suppose moreover that the wells are separated
\begin{align}
\label{eq:dist}
\min\limits_{i\neq j}\dist(SO(n)U_i, SO(n)U_j) \geq d>0.
\end{align}

\subsubsection{Triangulation}
\label{sec:grid}

Let $\Omega \subset \R^n$ be a bounded open Lipschitz domain. Suppose that
$\mathcal{T}_m=\bigcup\limits_{\alpha} T_{\alpha}$ is a non-degenerate triangulation
of $\Omega$, i.e. assume that for each $\alpha \in I_m$ the tetrahedron $T_{\alpha}
\subset \Omega$ is non-degenerate. More precisely, suppose that there are
constants $\tilde{c}_{1,T}, \tilde{c}_{2,T}>0$, which are independent of $m$, such
that for each $\alpha \in I_{m}$
\begin{align*}
\tilde{c}_{1,T} \leq r_{i,\alpha} \leq \diam(T_{\alpha}) \leq \tilde{c}_{2,T},
\end{align*}
where $r_{i,\alpha}, \diam(T_{\alpha})$ denote the in-radius and the diameter of the
tetrahedron $T_{\alpha}\in \mathcal{T}_m$.
In particular this implies that there are constants $c_{1,T}, c_{2,T}>0$ such that
for all $T_{\alpha}\in \mathcal{T}_m$
\begin{align}
\label{eq:non_deg}
c_{1,T} m^{-n} \leq |T_{\alpha}| \leq c_{2,T}m^{-n},
\end{align}
where $|T_{\alpha}|$ denotes the Lebesgue measure of $T_{\alpha}$.
Further assume that the triangulation $\mathcal{T}_m$ is incompatible with the
rank-one connections in $K$, i.e. there exists $\delta_0>0$ (which is independent of
$m$) such that for any $T_{\alpha}\in \mathcal{T}_m$ and for any normal $b \in
S^{n-1}$ associated to an $(n-1)$-dimensional interface of $T_{\alpha} $ we have
that
\begin{align}
\label{eq:incomp}
|b\cdot b_{ij}^{\pm}|\leq 1- \delta_0.
\end{align}
Here $b_{ij}^{\pm} \in S^{n-1}$ denotes any normal vector from \eqref{eq:rank_one}. 
In particular, by compactness of the involved sets this implies that there exists a constant $\overline{d}>0$ (depending on $\delta>0$ and $d$) such that for all choices of $i_1,i_2\in\{1,\dots,k\}$ with $i_1\neq i_2$
\begin{align}
\label{eq:d}
\min\limits_{Q\in SO(n)} \min\limits_{\substack{b \in S^{n-1} \text{ satisfying } \eqref{eq:incomp} \\ \text{ for all } i,j\in\{1,\dots,k\}}}  \min\limits_{\substack{\tau_1,\dots, \tau_{n-1} \in S^{n-1}, \\ \tau_{l}\cdot b=0,  \\ \tau_{l_1}\cdot \tau_{l_2}=0, \\ l_1 \neq l_2}} \max\limits_{\tau_l \in \{\tau_1,\dots,\tau_{n-1}\}}|(U_{i_1} - Q U_{i_2}) \tau_l| \geq \overline{d}>0.
\end{align}
Indeed, by compactness this minimum is attained. It does not vanish, as this would else yield a rank-one connection between $SO(n)U_{i_1}$ and $SO(n)U_{i_2}$, which is different from the ones listed in \eqref{eq:rank_one} in Section \ref{sec:wells}.

\begin{rmk}
\label{rmk:lower}
We note that by the assumption \eqref{eq:non_deg} the constant $\overline{d}$ from
\eqref{eq:d} has a strictly positive lower bound which is \emph{independent} of $m$.
\end{rmk}

\subsubsection{Deformation and energy}
\label{sec:energy}

Let $\mathcal{A}_m$ denote the set of piecewise affine deformations $u\in
W^{1,\infty}(\Omega, \R^n)$ adapted to the grid $\mathcal{T}_m$, i.e. for each
$\alpha \in I_m$ we have $\nabla u|_{T_{\alpha}} = const$.

For $u \in \mathcal{A}_m$ we then consider the energy
\begin{align}
\label{eq:e}
E_m(u) = \int\limits_{\Omega} h(\nabla u) dx,
\end{align}
where the energy density $h: \R^n \rightarrow [0,\infty)$ satisfies  the
comparability condition
\begin{align}
\label{eq:compare}
c_1 \dist^2(M, K) \leq h(M)  \mbox{ for all } M \in \R^{n\times n}_+
\end{align}
for some constant $c_1 >0$.
We in addition restrict our attention to sequences of deformations $\{u_m\}_{m\in
\N}\subset \mathcal{A}_m$ which are of \emph{surface energy scaling}, i.e. we assume
that there exists a constant $C>1$ (which does not depend on $m\in \N$) such that
\begin{align}
\label{eq:energy}
E_{m}(u_m)\leq C m^{-1}.
\end{align}

\subsubsection{Remarks on the assumptions from Sections
\ref{sec:wells}-\ref{sec:energy}}
Let us comment on the assumptions from Sections \ref{sec:wells}-\ref{sec:energy}: 
\begin{itemize}
\item Section \ref{sec:wells} asserts that we are considering a non-degenerate
multi-well problem, which for instance arises in the modelling and analysis of the
stress-free deformations of shape-memory materials. We remark that in this context,
we could also have added wells with only one or no rank-one connection to the other
wells. As our main motivation however stems from microstructures allowing for the
\emph{presence} of interfaces, we do not pursue this in the sequel. 
\item The conditions in Section \ref{sec:grid} render our problem a discrete
problem, since in Section \ref{sec:energy} we only consider piecewise affine
deformations adapted to the triangulation. The set of admissible grids is quite
large; in particular we do not require the underlying triangulations to be periodic
and only ask for a mild non-degeneracy in \eqref{eq:non_deg}. Requiring the
condition \eqref{eq:incomp} is however crucial for our purposes since it implicitly
provides a ``surface energy" regularizing contribution in the energy \eqref{eq:e}
(this ``finite element regularization" was observed for perhaps the first time in
the work of Lorent \cite{L09}, c.f. also the references therein). 
\item The energy densities $h$ which we consider are also only subject to mild
restrictions: The requirement in \eqref{eq:compare} ensures that in tetrahedra with
low energy the deformation gradient is close to the energy wells from Section
\ref{sec:wells}. Finally, the condition \eqref{eq:energy} ensures that we only
consider deformations with the simplest possible microstructures, i.e. we only
consider deformations of which we expect that they are very close to being exact
solutions to the differential inclusion $\nabla u \in K$. The smallness condition in \eqref{eq:energy} is such that it for instance excludes
microstructures which display branching phenomena.
\end{itemize}

\subsection{The main result}
\label{sec:result}
We assume that the conditions from Section \ref{sec:setup} hold.
Motivated by discrete-to-continuum limits as in \cite{KLR14}, \cite{KLR15}, \cite{ALP17}, we then seek
to deduce the following convergence and structure result:

\begin{thm}
\label{prop:BV_energy}
Let $n\geq 2$ and assume that $U_1,\dots,U_k$, $K$, $E_m$ are as in Section
\ref{sec:setup}.
Let $\{ u_m\}_{m\in \N} \subset \mathcal{A}_m$ be a sequence satisfying
\eqref{eq:energy}. Then, there exist (up to null-sets) disjoint Caccioppoli
partitionings
\begin{align*}
\Omega = \bigcup\limits_{j=1}^{k}\Omega_j, \quad \Omega_j =
\bigcup\limits_{i=1}^{\infty}\Omega_{j,i},
\end{align*}
with underlying characteristic functions $\chi_j, \chi_{j,i}: \Omega \rightarrow
\{0,1\}$, which in particular satisfy
\begin{align*}
\sum\limits_{j=1}^{k} \chi_j=1, \quad \sum\limits_{i=1}^{\infty}\chi_{j,i}=\chi_j,
\quad
\sum\limits_{j=1}^{k} \sum\limits_{i=1}^{\infty} |D \chi_{j,i}|(\Omega)<\infty,
\end{align*} 
a deformation $u\in W^{1,\infty}(\Omega)$ with $\nabla u \in BV(\Omega,K)$ and
rotations $R_{j,i}$ such that
\begin{align*}
&\nabla u_m \rightarrow \nabla u \mbox{ in } L^2(\Omega), \quad
\nabla u = \sum\limits_{j=1}^{k} \sum\limits_{i=1}^{\infty}R_{j,i} U_{j} \chi_{j,i} .
\end{align*}
\end{thm}

Here and in the sequel we use the convention that we denote the full (distributional) derivative of a ($BV$) function $\chi$ by $D \chi$ and use $\nabla \chi$ for its absolutely continuous part. The total variation measure of a $BV$ function is denoted by $|D\chi|(\cdot)$.\\

As in \cite{ALP17}, Lemma 4.2, in addition to the structure result of Theorem \ref{prop:BV_energy} one can show that the jump interfaces of $\nabla u$
are locally flat and can only intersect in ``corners" (i.e. lower dimensional
objects).\\

The main idea of the proof of Theorem \ref{prop:BV_energy} is to combine the
``spin-argument" from Lemma 2.1 in \cite{KLR15} (which we briefly recall in Section
\ref{sec:spin}), which was simultaneously also derived in \cite{ALP17}, with the
one-well, incompatible supercritical rigidity estimate from \cite{LL16} (c.f. also
\cite{LL17}). Here the reduction of the multi-well problem to an (incompatible)
auxiliary one-well problem is of particular interest and constitutes the main
novelty of the proof. We repeat that while in the context of the modelling and analysis of shape memory alloys the compactness result itself is not new, the generality in which it holds (c.f. the physical models covered in Section \ref{sec:gen})
and the argument for this result are new and of interest in themselves.

\section{The ``Spin-Argument"}
\label{sec:spin}

For self-containedness and completeness, we briefly recall the ``spin-argument" from
\cite{KLR15} and \cite{ALP17}. As a slight extension with respect to the argument
from \cite{KLR15} we directly prove it in arbitrary dimension.

\begin{lem}
\label{lem:incomp}
Let $n\geq 2$.
Let $\mathcal{T}_m$ be as in Section \ref{sec:grid} and assume that $T_{i_1},
T_{i_2} \in \mathcal{T}_m$ are adjacent grid tetrahedra, i.e. assume that $T_{i_1}$
and $T_{i_2}$ have a common $(n-1)$-dimensional interface with interface normal
$b\in S^{n-1}$. Let 
\begin{align}
\label{eq:c_0}
c_0:= \min\{\bar{d},d\},
\end{align}
where $d, \bar{d}>0$ are the constants from \eqref{eq:dist} and \eqref{eq:d}.
Suppose further that for $u\in \mathcal{A}_m$ 
\begin{align}
\label{eq:close_a}
|\nabla u|_{T_{i_1}} - U_{i_1}| \leq \frac{c_0}{100 },
\end{align}
but 
\begin{align}
\label{eq:far_a}
|\nabla u|_{T_{i_2}} - Q U_{i_1}| > \frac{c_0}{100 }
\end{align}
for all $Q \in SO(n)$. Then, 
\begin{align}
\label{eq:far}
\dist(\nabla u|_{T_{i_2}}, K) > \frac{c_0}{100 }.
\end{align}
\end{lem}

\begin{proof}
We show that for any $j\in\{1,\dots,k\}\setminus \{i_1\}$ and for any $Q\in SO(n)$
we have
\begin{align}
\label{eq:toshow}
|\nabla u|_{T_{i_2}}-Q U_{j}| > \frac{c_0}{100 },
\end{align}
where $c_0>0$ (c.f. Remark \ref{rmk:lower}) is the constant from \eqref{eq:c_0}.
By our assumption \eqref{eq:far_a} the claim of Lemma \ref{lem:incomp} follows, once
\eqref{eq:toshow} is shown.

To this end we note that as $T_{i_1}, T_{i_2}$ are neighbouring grid tetrahedra,
they share a common $(n-1)$-dimensional interface with normal $b\in S^{n-1}$. In
particular,
\begin{align}
\label{eq:facet}
(\nabla u|_{T_{i_1}}-\nabla u|_{T_{i_2}})\tau = 0,
\end{align}
for all $\tau \in S^{n-1}$ with $\tau\cdot b = 0$. By virtue of \eqref{eq:close_a}
and \eqref{eq:facet} we infer that
\begin{align}
\label{eq:facet1}
\max\limits_{\substack{\tau \in S^{n-1}, \ \tau \cdot b =0 }}|(U_{i_1}-\nabla u|_{T_{i_2}})\tau| \leq \frac{c_0}{100 }.
\end{align}
Now assuming that \eqref{eq:toshow} was wrong, we would obtain the existence of
$j_0\in\{1,\dots,k\}\setminus \{i_1\}$ and $\bar{Q}\in SO(n)$ such that
\begin{align}
\label{eq:contra}
|\nabla u|_{T_{i_2}}- \bar{Q}U_{j_0}|\leq \frac{c_0}{100}.
\end{align}
This however yields a contradiction: Indeed, by definition of $c_0$ and $\bar{d}$ (c.f. \eqref{eq:d})
\begin{align*}
|\nabla u|_{T_{i_2}}- \bar{Q} U_{j_0}|
&\geq  |(U_{i_1}-\bar{Q} U_{j_0})\tau_0| - \max\limits_{\tau \in S^{n-1}, \ \tau \cdot b = 0} |(\nabla u|_{T_{i_1}}-U_{i_1})\tau|\\
& \quad - \max\limits_{\tau \in S^{n-1}, \ \tau \cdot b = 0}  |(\nabla u|_{T_{i_1}}-\nabla u|_{T_{i_2}})\tau|
\\
& \geq c_0 -  \frac{c_0}{100 }  - \frac{c_0}{100 } > \frac{c_0}{100 }.
\end{align*}
Here the vector $\tau_0 \in S^{n-1}$, $\tau_0 \cdot b =0$ is chosen such that the inequality \eqref{eq:d} holds.
This yields contradiction to \eqref{eq:contra}, which hence concludes the argument
for the lemma.
\end{proof}

\begin{lem}
\label{lem:count}
Let $n\geq 2$, and assume that the conditions from Section \ref{sec:setup} hold.
Then,
\begin{align}
\label{eq:count}
\#\left\{ \alpha \in I_m: \dist^2(\nabla u_m |_{T_{\alpha}},K) \geq
\frac{c_0^2}{100^2} \right\} \leq \frac{C 100^2}{ c_0^2 c_{1}c_{1,T}} m^{n-1}.
\end{align}
\end{lem}

\begin{proof}
The claim follows from the energy estimate \eqref{eq:energy} and a counting
argument. Indeed, let $N:=\#\{  \alpha \in I_m: \dist^2(\nabla u_m |_{T_{\alpha}},K)
\geq \frac{c_0^2}{100^2} \}$. Then, by \eqref{eq:energy}, \eqref{eq:compare} and
\eqref{eq:non_deg}
\begin{align*}
N c_1 100^{-2} c_0^2  c_{1,T} m^{-n}
\leq N \dist^2(\nabla u_m |_{T_{\alpha}},K)\inf\limits_{\alpha}|T_{\alpha}| 
 \leq E_m(u) \leq  C m^{-1} 
\end{align*}
Solving for $N$ yields the claim.
\end{proof}

\begin{prop}
\label{lem:spin_arg}
Let $n\geq 2$ and suppose that the conditions from Section \ref{sec:setup} are
valid. Assume that $\{u_m\}_{m\in \N}$ is a sequence of deformations satisfying
\eqref{eq:energy}. Let further for $j\in\{1,\dots,k\}$
\begin{align*}
\Omega_{j,m}&:=\{T_{\alpha}\in \mathcal{T}_m: \ \dist(\nabla
u_m|_{T_{\alpha}},SO(n)U_j) \leq \frac{c_0}{100} \},\\
\Omega_{b,m}&:=\{T_{\alpha}\in \mathcal{T}_m: \ \dist(\nabla u_m|_{T_{\alpha}},K) >
\frac{c_0}{100} \}.
\end{align*}
Then there exist Caccioppoli sets $\Omega_1,\dots,\Omega_k \subset \Omega$ such that
\begin{align*}
\chi_{\Omega_j,m} \rightarrow \chi_{\Omega_j}, \ \chi_{\Omega_{b,m}} \rightarrow
0\mbox{ in } L^1(\Omega), \quad \Omega = \bigcup\limits_{j=1}^{k}\Omega_j.
\end{align*}
Here $\chi_{\Omega_{j,m}}, \chi_{\Omega_{b,m}}, \chi_{\Omega_j}:\Omega \rightarrow
\{0,1\}$ denote the characteristic functions associated with the corresponding sets.
\end{prop}

\begin{proof}
The proof follows from the Lemmas \ref{lem:incomp} and \ref{lem:count}:
Indeed, 
by Lemma \ref{lem:incomp} for $j\in\{1,\dots,k\}$
\begin{align*}
\Per(\Omega_{j,m}) &\leq \#\{T_{\alpha}\in \mathcal{T}_m: \ \dist(\nabla
u|_{T_{\alpha}},SO(n)U_j) > \frac{c_0}{100} \} \max\limits_{T_{\alpha}\in
\mathcal{T}_m}|\partial T_{\alpha}|\\
& \leq  \frac{C 100^2}{ c_0^2 c_{1}c_{1,T}} m^{n-1} 3 \tilde{c}_{2,T} m^{1-n} \leq
\frac{3 C 100^2 \tilde{c}_{2,T}}{ c_0^2 c_{1}c_{1,T}}.
\end{align*}
Here we used that at boundary tetrahedra, Lemma \ref{lem:incomp} asserts that the
local energy is larger than $\frac{c_0}{100}$. Hence, in particular, 
$|D \chi_{\Omega_{j,m}}|(\Omega) \leq \frac{3 C 100^2 \tilde{c}_{2,T}}{ c_0^2
c_{1}c_{1,T}}$, which yields that along a subsequence $\chi_{\Omega_{j,m}} \rightarrow \chi_{\Omega_j} $
in $L^1(\Omega)$ and $\chi_{\Omega_j} \in BV(\Omega)$. As $\Per(\Omega_{b,m})\leq \sum\limits_{j=1}^{k}\Per(\Omega_{j,m})$ we thus also obtain a similar uniform perimeter bound for $\Omega_{b,m}$ and therefore $\chi_{\Omega_{b,m}} \rightarrow \Omega_{b}$ along a further subsequence. 
Since $\Omega=\Omega_{b,m}\cup\bigcup\limits_{j=1}^{k}\Omega_{j,m}$, the sets $\Omega_{b,m}, \Omega_{1,m},\dots,\Omega_{k,m}$ hence form a Caccioppoli partitioning of $\Omega$ with a uniform (in $m$) perimeter bound. As a consequence, $\Omega_b, \Omega_1,\dots, \Omega_k$ also form a Caccioppoli partitioning of $\Omega$ (c.f. Theorem 4.19 in \cite{AFP}).

Finally, we note that $|\Omega_{b}|=0$, as by Lemma \ref{lem:count} we have that
\begin{align*}
|\Omega_b| 
&\leq \#\{T_{\alpha}\in \mathcal{T}_m: \ \dist(\nabla u|_{T_{\alpha}},SO(n)U_j) >
\frac{c_0}{100} \} \max\limits_{T_{\alpha}\in \mathcal{T}_m}|T_{\alpha}|\\
&\leq \frac{C 100^2}{ c_0^2 c_{1}c_{1,T}} m^{n-1} c_{1,T} m^{-n} \\
&= C 100^2 (c_0^2 c_1 )^{-1} m^{-1} \rightarrow 0 \mbox{ as } m \rightarrow \infty.
\end{align*}
This concludes the argument.
\end{proof}

\section{Proof of Theorem \ref{prop:BV_energy}}
\label{sec:proof}

\subsection{A compactness result for incompatible fields close to rotations}

With the result of Lemma \ref{lem:spin_arg} at hand, we approach the proof of our
main result. In this context, we will frequently use the following notation: We will often identify vector fields $v \in L^2(\Omega,\R^n)$ with co-vectorfields $\omega= v^{\flat} = \sum\limits_{j=1}^n v_j dx_j \in L^2(\Omega, \Lambda^1)$. If $A\in L^2(\Omega, \R^{n\times n})$ we correspondingly identify it with $\omega \in L^2(\Omega,\Lambda^1)^n$. 
With slight abuse of notation for $v \in C^{\infty}_0(\Omega,\R^n)$ we further do not distinguish between $\Curl(v)=(* dv^{\flat})^{\sharp}$ and the two-form $d v^{\flat}=d \omega$, where
\begin{align*}
d v^{\flat} = \sum\limits_{j<k} \left( \frac{\partial v^j}{\partial x_k} - \frac{\partial v^k}{\partial x_j} \right) dx^j \wedge dx^k.
\end{align*}
If $v$ is less regular, we interpret $\frac{\partial v^j}{\partial x_k} - \frac{\partial v^k}{\partial x_j} $ distributionally. If $\frac{\partial v^j}{\partial x_k} - \frac{\partial v^k}{\partial x_j} $ is a Radon measure, we write $|\Curl(v)|(\cdot)$ to denote its total variation. We denote the set of bounded Radon measures on $\Omega$ with values in $\Lambda^k$ by $\mathcal{M}_b(\Omega,\Lambda^k)$. Analogously, we use the notation $\mathcal{M}_b(\Omega,\Lambda^k)^n$ for vectors of bounded Radon measures.
\\

With this preparation at hand, we recall a consequence of the estimates in \cite{LL16},
which will be of central relevance to us:

\begin{prop}[Proposition 3 in \cite{LL17}]
\label{prop:LL_P3}
Let $\Omega \subset \R^n$, $n\geq 2$, be a bounded, open, simply connected set and
consider a sequence of matrix fields $A_j \in L^2(\Omega, \R^{n\times n})$ such that
\begin{align*}
\lim\limits_{j \rightarrow \infty} \|\dist(A_j, SO(n))\|_{L^2(\Omega)} = 0, \quad
\sup\limits_{j\in \N} |\Curl(A_j)|(\Omega)\leq C,
\end{align*}
where the operator $\Curl$ is understood as explained above.

Then, up to subsequences, $\{A_j\}_{j\in\N}$ converges strongly in $L^2(\Omega)$ to
a matrix field $A\in BV(\Omega, SO(n))$ and 
\begin{align*}
|D A|(\Omega) \leq C |\Curl(A)|(\Omega).
\end{align*}
\end{prop}

For the convenience of the reader, we include here a sketch of the proof of Proposition~\ref{prop:LL_P3}. The starting point is the identity

\begin{equation}
    \label{eq:loo}
	T\de \omega + \de T\omega = \omega,
\end{equation}
where $T$ is the \emph{averaged linear homotopy operator} (cf.~\cite{LL17}), which holds for every differential form $\omega \in L^1(B, \Lambda^r)$ on the unit ball $B = B(0, 1) \subset \mathbb{R}^n$, whose exterior derivative is a bounded Radon measure $\de \omega \in \mathcal{M}_b(B, \Lambda^{r+1})$. Since it can be easily seen that $T$ is a ``weakly singular'' operator (that is it maps $L^{\infty}$ continuously into itself), the weak geometric rigidity estimate~\cite[Corollary 4.1]{CDM14} gives the following \emph{weak critical (i.e. for the exponent $1^*:=\frac{n}{n-1}$) geometric rigidity estimate for incompatible fields} (cf.~\cite[Theorem 3]{LL17}):

\begin{thm}
 \label{thm:weak_critical}
	Let $B:=B(0, 1) \subset \mathbb{R}^n$ with $n\geq 2$. There exists a constant  $C = C(n) > 0$, which depends only on the dimension $n$ such that for every matrix field $A \in L^{1^*, \infty}(B, \R^{n\times n})$ whose distributional $\Curl$ is a bounded Radon measure, i.e. $\de A \in \mathcal{M}_b(B, \Lambda^2)^n$ there exists an associated rotation $R\in SO(n)$ such that
	\begin{equation}
	 \label{eq:weak_critical}
	 \norm{A - R}_{L^{\frac{n}{n-1}, \infty}(B, \R^{n\times n})} \le C\rB{\norm{\dist(A, SO(n))}_{L^{\frac{n}{n-1}, \infty}(B, \R^{n\times n})} + \abs{\Curl(A)}(B)}.
	\end{equation}
\end{thm}
Moreover, it can be checked that $\nabla T$ is the sum of a weakly singular operator and a Calder\'on-Zygmund operator. This remark gives, after straightforward computations, the following (cf.~\cite[Theorem 4]{LL17}):

\begin{thm}
 \label{thm:strong_supercritical}
 Let $B:=B(0, 1) \subset \mathbb{R}^n$ and $p > 1^*(n)$. There exists a constant  $C = C(n, p) > 0$, which depends only on the dimension $n$ and the exponent $p$, such that for every matrix field $A \in L^p(B, \R^{n\times n})$ such that $\de A^{\flat} \in \mathcal{M}_b(B, \Lambda^2)^n$, there exists an associated rotation $R\in SO(n)$ such that
 \begin{equation}
  \label{eq:strong_supercritical}
  \norm{A - R}_{L^p(B, \R^{n \times n})}^{p} \le C\cB{\norm{\dist(A, SO(n))}_{L^p(B, \R^{n \times n})}^{p} + \abs{\Curl(A)}^{\frac{n}{n-1}}(B)}
 \end{equation}
\end{thm}

Clearly, the constants in~\eqref{eq:weak_critical} and~\eqref{eq:strong_supercritical} are scaling invariant. With Theorem~\ref{thm:weak_critical} and Theorem~\ref{thm:strong_supercritical} at hand one can easily prove Proposition~\ref{prop:LL_P3} (cf. also~\cite[Proposition 1]{LL17}). Indeed, on the one hand, Theorem~\ref{thm:strong_supercritical} gives compactness with respect to the strong $L^2$ topology, i.e., up to a subsequence, $A_j \to A \in L^2(\Omega, \R^{n \times n})$ strongly in $L^2(B, \R^{n \times n})$. The argument for this follows in two steps (c.f. \cite{LL16}, proof of Proposition 3): First, weak $L^2$ convergence is used to pass to the limit in \eqref{eq:strong_supercritical} yielding that for every $x\in \R^n, \rho>0$ there exists $R_{\rho,x}\in SO(n)$ such that
\begin{align*}
\|A- R_{\rho,x}\|_{L^2(B_{\rho,x},\R^{n\times n})}^2
\leq C T(B_{\rho,x})^{\frac{n}{n-1}},
\end{align*}
where $T$ is the weak-$*$-limit of $|\Curl(A_j)|(\cdot\cap B_{\rho}(x))$. In a second step, a covering argument, in which the scaling invariance of the original inequality \eqref{eq:strong_supercritical} is crucial, shows that 
\begin{align*}
\Dim_{\mathcal{H}}(\{x\in \Omega: A\notin SO(n)\}) \leq n-1.
\end{align*}
On the other hand, the $BV$ structure of the limit field $A$ can be deduced using Theorem~\ref{thm:weak_critical}. Indeed, one can always approximate the field $A$ with a piecewise constant one of the form
\[
 A_{\rho}:=\sum_{i} R_i^{(\rho)} \chi_{Q_{i, \rho}},
\] 
where the cubes $Q_{i, \rho} = Q(x_i, \rho)$ of side length $\rho > 0$ (whose interiors are mutually disjoint) define a tessellation of $\mathbb{R}^n$, the sum is extended over those cubes which intersect the domain $\Omega$ and the rotations $R_i^{(\rho)}$ are given by Theorem~\ref{thm:weak_critical} applied to the balls $B(x_i, \frac{3}{2}\rho)$. It is then easy to estimate the total variation of $A_{\rho}$ in terms of $\de A$, using the fact that the weak-$L^p$ norm is comparable to the strong-$L^p$ norm for constant functions.

\subsection{Proof of Theorem \ref{prop:BV_energy}}

Relying on Proposition \ref{prop:LL_P3}, we present the proof of Theorem
\ref{prop:BV_energy}:

\begin{proof}[Proof of Proposition \ref{prop:BV_energy}]
\emph{Step 1: Truncation.} Using a truncation argument (c.f. for instance
\cite{FJM02}), we may without loss of generality assume that $ u_{m} \in
W^{1,\infty}(\Omega)$. Indeed, if this were not the case, it would always be
possible to replace the sequence $u_m$ by a sequence $v_m$ with the property that
for a constant $\bar{c}>0$ which only depends on $\Omega,n$
\begin{align*}
\|\nabla v_m\|_{L^{\infty}(\Omega)} \leq \bar{c} 100 d, \
\|\nabla u_m - \nabla v_m\|_{L^2(\Omega)}^2 \leq \bar{c} \int\limits_{\{|\nabla u_m|
\geq 100 d\}}|\nabla u_m|^2 dx \leq \bar{c}\frac{C}{m}.
\end{align*}
Here we used the energy bound \eqref{eq:energy} to infer the last estimate.
In particular, $v_m$ satisfies the energy bounds of the same type as $u_m$.
Therefore, in the sequel, we always assume that we already have that $u_{m} \in
W^{1,\infty}(\Omega)$.
\\

\emph{Step 2: Reduction to the one-well problem.}
Invoking the energy estimate \eqref{eq:energy} and the bound \eqref{eq:compare} we
obtain
\begin{align}
\label{eq:close_elast}
&\int\limits_{\Omega} \dist^2 \left(\sum\limits_{j=1}^{k} \chi_{j,m} \nabla u_{m}
U^{-1}_j ,SO(n)\right) dx \leq C c_1^{-1} m^{-1},
\end{align}
where for $j\in\{1,\dots,k\}$ the functions $\chi_{j,m}:=\chi_{\Omega_{j,m}}$ are
the characteristic functions from Lemma \ref{lem:spin_arg}.
Using the boundedness of $\nabla u_m$ and the convergence $\chi_{j,m} \rightarrow
\chi_j:= \chi_{\Omega_j}$, which was derived in Lemma \ref{lem:spin_arg}, implies
that the vector fields
\begin{align}
\label{eq:A_e}
A_{j,m}(x):=  \chi_j|_{U_j^{-1}x}  \nabla u_{m}|_{U_j^{-1}x}
U^{-1}_j,
\end{align}
satisfy
\begin{align}
\label{eq:est4}
& \int\limits_{\tilde{\Omega}_j} \dist^2(A_{j,m},SO(n)) dx \leq 2Cc_{1}^{-1} m^{-1},\\
\label{eq:est5}
& |\Curl(A_{j,m})| \leq C|D \chi_j| \mbox{
as measures, i.e.},  \\
& \notag |\Curl(A_{j,m})|(\Omega') \leq C|D \chi_j|(\Omega') \mbox{ for all } \Omega' \subset \Omega.
\end{align}
Here we set $\tilde{\Omega}_j = U_j \Omega_j$ and used that 
\begin{align*}
\Curl(\nabla u_{m}|_{U_j^{-1}x}
U^{-1}_j) = \Curl(\nabla u_{m}(U_j^{-1} \cdot ))=0.
\end{align*}
Indeed, \eqref{eq:est4} directly follows from \eqref{eq:close_elast}, while \eqref{eq:est5} follows from the computation of the distributional curl of $A_{j,m}$, a mollification argument (on the level of $u_m$) and the fact that $\nabla u_m \in L^{\infty}(\Omega)$, which is a consequence of the Lipschitz truncation from Step 1.  
\\

\emph{Step 3: Application of the one-well rigidity result and conclusion.}
Combining the estimates \eqref{eq:est4} and \eqref{eq:est5} with Proposition
\ref{prop:LL_P3}, we infer that $A_{j,m} \rightarrow A_j$ in $L^2(\Omega)$ with $A_j \in
SO(n)$ and
\begin{align}
\label{eq:grad_est}
|D A_j|(\Omega) \leq C|\Curl(A_j)|(\Omega) \leq C |D
\chi_j|(\Omega) .
\end{align}
Here we used \eqref{eq:est5} and the lower semicontinuity of the perimeter in order
to deduce the last estimate in \eqref{eq:grad_est}.
In particular, since by the energy estimate \eqref{eq:energy} there exists $u$ with
$\nabla u_{m} \rightharpoonup \nabla u$ in $L^{2}(\Omega)$, the strong $L^2$
convergence of $A_{j,m}$ yields that $\nabla u_{m} \rightarrow \nabla u$ in
$L^2(\Omega)$. 
We can therefore pass to the ($L^2(\tilde{\Omega})$-)limit in the identity
\eqref{eq:A_e} and infer that
\begin{align*}
A_{j}(x) =  \chi_j|_{U^{-1}_j x}\nabla u|_{U^{-1}_j x} U_j^{-1}.
\end{align*}
We note that due to \eqref{eq:grad_est} the associated jump sets satisfy $J_{A_j}
\subset J_{\chi_j\circ U_j^{-1}} $.
In particular this entails that $A_j \in SBV(\Omega, SO(n))$. Moreover, it implies
that on $\tilde{\Omega}_j \setminus \partial^{*} \tilde{\Omega}_j$ we have $\nabla A_j = 0$ and thus
$A_j=R_{j,i}$ for some $R_{j,i} \in SO(n)$, which is constant on each of the at most countably many BV indecomposable
components of $\tilde{\Omega}_j$ (c.f. Proposition 2.13 in \cite{DM1} or 4.2.25 in \cite{Fe}). As a consequence, there exist characteristic functions
$\tilde{\chi}_{j,i}, \chi_{j,i}$ and $R_{j,i}\in SO(n)$ with
\begin{align*}
A_j(x)=  \sum\limits_{i=1}^{\infty} \tilde{\chi}_{j,i}(x) R_{j,i} 
 =: \sum\limits_{i=1}^{\infty}   \chi_{j,i}(U_j^{-1} x) R_{j,i}.
\end{align*}
Using that this is valid for all $j\in\{1,\dots,k\}$ and rewriting it in terms of $\nabla u$ therefore results in
\begin{align*}
\nabla u (x)
=  \sum\limits_{j=1}^{k}\sum\limits_{i=1}^{\infty} \chi_{j,i}(x) R_{j,i}U_{j},
\end{align*}
where $ \sum\limits_{j=1}^{k}\sum\limits_{i=1}^{\infty} |D
\chi_{j,i}|(\Omega)\leq C  \sum\limits_{j=1}^{k}|D \chi_j|(\Omega)<\infty$. This
concludes the argument.
\end{proof}

\section{Generalizations}
\label{sec:gen}

In this section, we show that the arguments which were presented in Sections \ref{sec:setup_result}-\ref{sec:proof} generalize to a much larger class of physical systems. These include the martensitic phase transformations from Sections \ref{sec:setup_result}-\ref{sec:proof} as special cases. In order to achieve this degree of generality, we consider Hamiltonians which allow for a \emph{periodic} ground state structure, corresponding to settings in which there is an ``internal" microstructure in the different phases. In particular, the associated Hamiltonians can be much more ``nonlocal" and cover a significantly larger class of physically interesting phase transformations (c.f. Examples \ref{ex:anti_ferr} and \ref{ex:higher_dim_rot_inv} on anti-ferromagnetic transformations and on more general transformations with internal twinning structure).

\subsection{Setting}
\label{sec:Hamil}
We seek to generalize the structure result of Theorem \ref{prop:BV_energy} to a larger class of Hamiltonians. Let us be more precise about this: First,
adopting a microscopic point of view, we consider deformations $X_i: \Z^n \rightarrow \R^n$. 
Given a set $\hat{\Omega} \subset \Z^n$, we study Hamiltonians $H:\Z^n \rightarrow \R$ with a translation and rotation invariant energy density $h: \R^{n \times q} \rightarrow \R_+ $ such that
\begin{align}
\label{eq:Hamiltonian}
H(X_j) = \sum\limits_{j\in \Z^n : \ \tau_j(\Lambda) \subset \hat{\Omega}} h ( X_{j+i})_{i\in \Lambda}.
\end{align}
Here $\Lambda \subset \hat{\Omega}$ with $\max\{\# \Lambda, \diam(\Lambda)\} = q<\infty$ and $\tau_j(\Lambda):=\{i\in \Z^n: i-j \in \Lambda\}$. The notation $h (X_{j+i})_{i\in \Lambda}$ is used to abbreviate a dependence of $h$ on all the values $X_{j+r}$, $r\in \Lambda$.
Invoking the translation invariance and denoting the canonical basis of $\R^n$ by $e_1,\dots,e_n$, we observe that the Hamiltonian \eqref{eq:Hamiltonian} can be rewritten as a function of the discrete gradient,  
\begin{align*}
\nabla_d X_i:= (X_{i+e_1}-X_i,  \dots, X_{i + e_n} - X_i)^t,  \ i \in \Z^n,
\end{align*}
only (which we evaluate on a finite range of lattice points).
More precisely, by virtue of the translation invariance of the problem, the Hamiltonian takes the form
\begin{align*}
H(X_j)= \sum\limits_{j\in\Z^n: \ \tau_j(\Lambda) \subset \hat{\Omega}} \tilde{h}(\nabla_d X_{j+i})_{i\in \Lambda},
\end{align*}
where $\tilde{h}: \R^{n \times n \times q} \rightarrow \R_+$ is obtained from the energy density $h$. 
Also, using a piecewise affine interpolation on an underlying triangulation of $\Z^n$, we can always identify the discrete deformation $X_j: \Z^n  \rightarrow \R^n$ with its piecewise affine interpolation $X:\R^n \rightarrow \R^n$ (on a fixed triangulation of the lattice). On the underlying dual lattice triangles we in particular have $\nabla_d X_i = \nabla X(i)$. In the sequel, we will often switch between the discrete and continuous viewpoints without further comment.

In deriving our main structure result, we will further suppose that the following conditions hold: 

\begin{itemize}
\item[(H1)] The ground states (or rather their gradients) are periodic. More precisely, there exist deformations $Z_{l,j}: \Z^n \rightarrow \R^{n}$, $l\in \{1,\dots,k\}$, such that the following conditions hold:
\begin{itemize}
\item[(a)] The gradients $\nabla_d Z_{l,j}$ are periodic functions with a rectangular period cell $Q_{l} \subset \Z^n$ of finite diameter $\diam(Q_l)\leq L_l<\infty$. We set $L_0:= \max\limits_{l\in\{1,\dots,k\}}\{\diam(Q_l)\}$ and $Q_{0}:= [0,L_0]^n$.
\item[(b)] We have
\begin{align*}
&\sum\limits_{j\in \Z^n} h(X_{j+i})_{i\in \Lambda}= 0\\
&\Leftrightarrow  X_j = R Z_{l,j} + b \mbox{ for some } l\in\{1,\dots,k\}, \ R \in SO(n), \ b \in \R^n \\
& \quad \quad \mbox{ and all } j \in \Z^n. 
\end{align*}
\item[(c)] For each $l\in \{1,\dots,k\}$ the averaged gradients
\begin{align*}
U_l:= |Q_l|^{-1}\int\limits_{Q_l} \nabla_d Z_{l}(x)dx
\end{align*}
are invertible. In particular, the gradients of the ground states can be split into a mean deformation (which averages out the microscopic oscillations) and an oscillatory part with average zero:
\begin{align*}
\nabla_d Z_l = U_l + (\nabla_d Z_l - U_l).
\end{align*}
\item[(d)] The ground states are incompatible in the sense that  
\begin{align}
\label{eq:dist_H}
\dist(SO(n)\nabla_d Z_{l_{1},i}, SO(n)\nabla_d Z_{l_2,i})|_{i \in Q_{0}}\geq d >0,
\end{align}
where for two mappings $Y_1(i),Y_2(i)$, a set $B \subset \Z^n$ and a point $j_0 \in  \Z^n$ we set $\dist(Y_1(i), Y_2(i))|_{i\in B + j_0}:= \max\limits_{i\in B+j_0}\dist(Y_1(i), Y_2(i))$.
\end{itemize}

\item[(H2)] There exists $p\in (1,\infty)$ and a box $\tilde{\Lambda}$ which strictly contains $\Lambda$ and $2 Q_{0}:=[0,2 L_0]$, i.e. $\Lambda \subset \tilde{\Lambda}$ and $2Q_{0} \subset \tilde{\Lambda}$, such that if for all $R\in SO(n)$, for some $j_0\in \Z^n$ with $\tau_{j_0}(\tilde{\Lambda}) \subset \hat{\Omega}$ and $\kappa \in \R_+$ it holds that 
\begin{align*}
\dist(R \nabla_d Z_{l,i}, \nabla_d X_i)|_{i\in \tilde{\Lambda} + j_0} \geq \kappa,
\end{align*}
then it follows that
\begin{align*}
\sum\limits_{j\in \Z^n: \ \tau_j(\Lambda) \subset \tilde{\Lambda}+j_0} h(X_{j+i})_{i\in \Lambda} \geq c \kappa^p .
\end{align*}
Here $\dist(R \nabla_d Z_{l,i}, \nabla_d X_i)|_{i\in \tilde{\Lambda} + j_0}:= \max\limits_{i\in \tilde{\Lambda}+j_0}\dist(R \nabla_d Z_{l,i}, \nabla_d X_i)$. 
\end{itemize}

Let us comment on these assumptions: The Hamiltonian \eqref{eq:Hamiltonian} is allowed to depend on a quite large range of values, it is significantly more ``nonlocal" than the Hamiltonian which was considered in our model set-up in Section \ref{sec:setup}. Moreover, we emphasize that the Hamiltonians presented here are not only tailored to cover martensitic phase transitions, but also allow for other classes of classical phase transformations, including for example anti-ferromagnetic ones (c.f. Examples \ref{ex:anti_ferr}, \ref{ex:higher_dim_rot_inv}).
The conditions (H1)-(H2) ensure that they nevertheless mathematically display similar structural properties as the Hamiltonian from Section \ref{sec:energy}:
\begin{itemize} 
\item The first condition (H1) determines the energy wells of the Hamiltonian and implies that it can for instance be used to describe martensitic phase transformations. However, by assuming that (the gradients of) the ground states are periodic instead of being constant, we also cover a number of other interesting phase transformations, in particular it can be used to describe \emph{phase-antiphase} boundaries. A typical system, which for instance is included, is an anti-ferromagnetic Ising type model. 

The invertibility condition which is stated in (H1)(c) is a mathematical artefact of our proof. We however emphasize that we only require the \emph{averages} $U_l$ to be invertible, in particular the full ground state deformation gradient $\nabla_d Z_{l,i}$ may include strong oscillations. The invertibility condition can be viewed as having density estimates from above and below.
 
For some simple physical systems the invertibility requirement can be relaxed by modifying the ground states: For instance, for one dimensional systems such as the anti-ferromagnetic spin system from Example \ref{ex:anti_ferr}, this can be achieved by passing from the deformations $Z_{l,j}$ to deformations $\widetilde{Z}_{l,j}:= Z_{l,j} + C j$, where $C \in \R_+$ is a large positive constant. Choosing $C>0$ sufficiently large, then ensures the invertibility of the corresponding averaged gradients $U_l$. 
For more complex physical systems with rotation invariance, the invertibility condition however cannot be recovered by such a simple argument in general.

\item Property (H2) implies that the Hamiltonian is bounded from below by a power of the distance function to the wells. The requirement $p\in (1,\infty)$ yields a growth condition for the Hamiltonian in the neighbourhoods of the wells. It also replaces Lemma \ref{lem:incomp}, which ensures that we necessarily catch interfacial energy, when switching between the wells. By a similar argument, it further allows us to average on microscopic scales and hence to ignore the (possibly large) oscillations on these scales.
The restriction to $p\in(1,\infty)$ is a mathematical artefact ensuring the validity of one-well rigidity estimates. 
\end{itemize}

\begin{example}[Anti-ferromagnetic spin model]
\label{ex:anti_ferr} 
An example of a phase transition to which our set-up applies is a (one-dimensional) anti-ferromagnetic spin Hamiltonian: We begin by discussing a nearest neighbour interaction. In order to show how this fits into our framework, we use a (for this model) slightly cumbersome notation (which is chosen in order to resemble the set-up layed out in the conditions (H1)-(H2)). Let $\Omega=[0,L]$ and set
\begin{align}
\label{eq:anti_ferr1}
H(X_j):= \sum\limits_{j\in \Z \cap [0,L-1]} [(\nabla_d X_j) (\nabla_d X_{j+1})+1],
\end{align}
where $X_{i}:\Z\cap \Omega \rightarrow \{\pm 1\}$ (which in particular entails that $\nabla_d X_i \in \{0,\pm 1\}$). Up to translations, the ground states of this Hamiltonian are given by the following two piecewise affine, periodic deformations $Z_{1,j}, Z_{2,j}$, which are determined by the requirements
\begin{align*}
\nabla_d Z_{1,j} = 1, \ \nabla_d Z_{1,j+1} = -1 \mbox{ or } \nabla_d Z_{2,j} = -1, \ \nabla_d Z_{2,j+1} = 1.
\end{align*}
We remark that a transition between $Z_{1,j}, Z_{2,j}$ necessarily costs interfacial energy as it involves two parallel spin vectors, c.f. Figure \ref{fig:arrows}. In particular, if an energy bound as in \eqref{eq:en_dis} below is assumed to hold, the number of transitions between the states $Z_{1,j}$ and $Z_{2,j}$ is uniformly controlled.

\begin{figure}
\centering

\begin{subfigure}[t]{.25\textwidth}
\centering
\includegraphics[width=\linewidth]{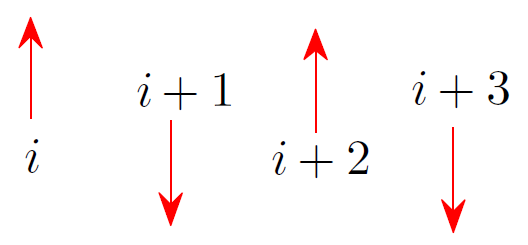}
       \caption{} \label{fig:fig_a} 
\end{subfigure} \hfill
\begin{subfigure}[t]{.25\textwidth}
\centering
\includegraphics[width= \linewidth]{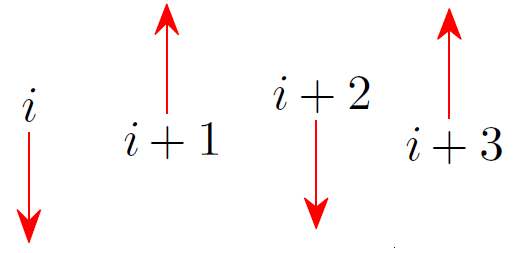}
        \caption{}\label{fig:fig_b}
\end{subfigure} \hfill
\begin{subfigure}[t]{.25\textwidth}
\centering
\includegraphics[width=\linewidth]{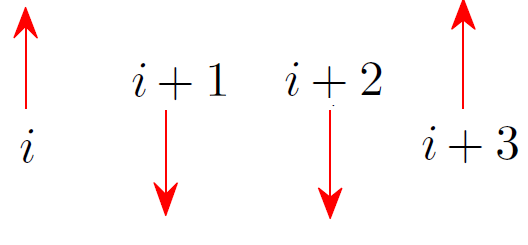}
        \caption{}\label{fig:fig_c}
\end{subfigure}

\caption{The two ground state configurations for the spin system from Example \ref{ex:anti_ferr} consist of a combination of an ``up" and a ``down" spin (illustrated over two periods in panel \eqref{fig:fig_a}) or a ``down" and an ``up" spin (illustrated over two periods in panel \eqref{fig:fig_b}). This is a typical example of a system with a phase and an antiphase displaying phase antiphase transitions.
If there is a transition between the two ground states (as for instance illustrated in panel (\ref{fig:fig_c})), this costs energy as for instance a ``down" spin is next to another ``down" spin.
}
\label{fig:arrows}
\end{figure}

While the conditions (H2) on the Hamiltonian can be checked to be satisfied, the invertibility requirement in condition (H1)(c) is violated (as a the average of (the gradient of) the sawtooth function is clearly not invertible). As the Hamiltonian however only depends on the finitely many values of the (discrete) gradient (we recall that $X_i \in \{\pm 1\}$), the problem \eqref{eq:anti_ferr1} is ``equivalent" to a setting in which the saw-tooth ground states are mapped to ground states which are invertible: More precisely, the ground states only involve the gradients $\pm 1$ (general deformations, which are not necessarily ground states, might also have $0$ as a third option for its gradient, but cannot attain more values, since the admissible deformations $X_i$ are constraint to attain only the values $\pm 1$). Hence, the anti-ferromagnetic spin model could have been mapped to a model in which the ground states attain the gradient values $1,2$. The additionally possible gradient value $0$ could be mapped to any number different from $1,2$. With this modification, the anti-ferromagnetic spin system is admissible in our framework, as now the invertibility constraint is also satisfied. In order to have an explicit setting in mind, we remark that an associated Hamiltonian could for instance be given by
\begin{align*}
\tilde{H}(\nabla X_j) = \sum\limits_{j\in \Z \cap [0,L-1]} [(\nabla_d X_j)-1][(\nabla_d X_{j+1})-2] + 1
\end{align*}
for the class of deformations, which only attain the values $\nabla_d X_j \in \{1,\frac{3}{2},2\}$.

We note that instead of considering the simple nearest neighbour anti-ferromagnetic spin Hamiltonian, we could also have considered more complicated anti-ferromagnetic Hamiltonians, which involve longer range interactions. For possible microstructures that arise in this more general framework we refer to \cite{ABC06}.
\end{example}

\begin{rmk}
\label{rmk:phases_shifts}
We observe that in the previous example the two ground states $Z_{1,j}, Z_{2,j}$ are physically not really different ``phases". As they only differ by a shift, they could be considered as being the same phase but in a different shifted form. In physical terms they correspond to antiphases, their interfaces are \emph{phase-antiphase boundaries}.
Mathematically, this could have been emphasized and formalized by working with two variables, one denoting the phase, one the shift. In order to avoid further technicalities, we have opted not to pursue this here, and have instead subsumed both the variants of the phases and the different phases in the collection $Z_{l,j}$ of the ground states (c.f. (H1)).
\end{rmk}

\begin{example}[Phases with internal microstructure]
\label{ex:higher_dim_rot_inv}
A higher dimensional example of the setting, which is covered by our class of Hamiltonians, is given by phase transformations with an ``internal microstructure". A model setting of this consists for instance of ground states which are themselves twinned, see Figure \ref{fig:boxes}. Here the different ``variants" consist of translations of the twinning structure. Moreover, rotations of the structures are also possible, these are however identified as corresponding to the same phase. A macroscopic state could combine several mesoscopic states, involving both different phases and different ``microphases". 

\begin{figure}
\includegraphics[scale=0.5]{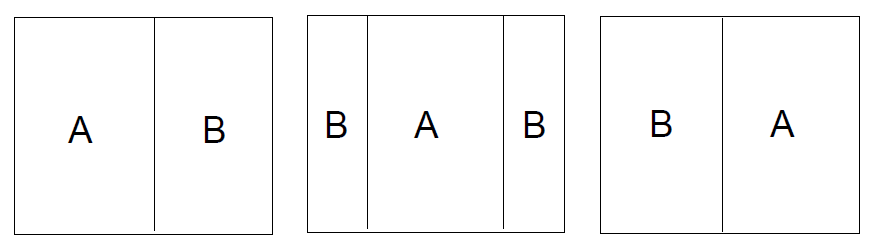}
\caption{An example of a ground state, which consists of a twin between two variants $A, B$. There are different ``realizations" of this, some examples are indicated in the three figures. Although these are physically only shifts of a single ``parent phase" and hence give rise to phase-antiphase transitions, we treat these mathematically as different phases for simplicity of notation.}
\label{fig:boxes}
\end{figure}
\end{example}

As we are interested in the limit from the microscopic to the macroscopic scales, we next introduce a lattice on a fixed domain $\Omega \subset \R^n$ and consider a rescaled version of the Hamiltonian from above. To formulate this more precisely, we use the following notation: For an arbitrary set $B \subset \R^n$ and any number $m\in \R_+$ we define 
\begin{align}
\label{eq:set_scale}
m B:=\{x\in \R^n: m^{-1}x \in B\}.
\end{align} 
With this notation at hand, we set $\hat{\Omega}:= m \Omega$ for each $m\in \N$ 
and work with
\begin{align*}
H_m(X_i):= \sum\limits_{j\in \Z^n: \ \tau_j(\Lambda) \subset \hat{\Omega}} m^{-n} \tilde{h}(\nabla_d X_{j + i})_{i \in \Lambda}.
\end{align*}
For a fixed triangulation of the underlying lattice, the piecewise affine interpolation at scale $m^{-1}$ associated with the microscopic mapping $X_i: \Z^n \rightarrow \R^n$ is then given by 
\begin{align}
\label{eq:interpol}
u_m(m^{-1}i):= m^{-1}X_{i}.
\end{align}

In addition to the conditions on the Hamiltonian, which have been explained in (H1)-(H2) from above, we will always assume a ``low energy condition", which essentially corresponds to surface energy scaling. 
To this end, we assume that any sequence of deformations $\{(X_{m})_i\}_{m\in \N}$, which will be considered in the following, satisfies the bound
\begin{align}
\label{eq:en_dis}
H_m(X_{m,i})\leq C m^{-1} 
\end{align}
for some $C>0$, which is independent of $m$. 

\subsection{The main result}
In the sequel, we seek to prove the following analogue of Theorem \ref{prop:BV_energy}, which also yields a corresponding structure result in the setting of the more general phase transformations discussed in the present section.

\begin{thm}
\label{thm:BV_struc_gen}
Let $\Omega \subset \R^n$ with $n\geq 2$ be an open, bounded Lipschitz domain.
Let $H_m$ be a Hamiltonian as described in Section \ref{sec:Hamil}, in particular assume that the conditions (H1)-(H2) hold. Further suppose that $\{(X_{m})_{i}\}_{m\in\N}$ with
\begin{align*}
(X_{m})_i: \Z^n \rightarrow \R^n \mbox{ for each } m \in \N,
\end{align*}
is a sequence of deformations satisfying \eqref{eq:en_dis} (with a uniform constant $C>0$). Denote for each $m\in \N$ by $u_m \in W^{1,\infty}_{loc}(\R^n, \R^n)$ the interpolation of $(X_m)_i$, which was defined in \eqref{eq:interpol}.
Then there exists a function $u \in BV(\Omega, \R^{n})$ and a Caccioppoli partitioning  
\begin{align*}
\Omega = \bigcup\limits_{l=1}^{k} \bigcup\limits_{j=1}^{\infty} \Omega_{l,j}, 
\end{align*}
with associated characteristic functions $\chi_{l,j}$ such that
\begin{itemize}
\item $u_m \rightharpoonup  u$ in $W^{1,p}_{loc}(\Omega,\R^{n})$ (where $p\in(1,\infty)$ denotes the exponent from (H2)) and $\nabla u \in BV_{loc}(\Omega)$, 
\item  there exist countably many $R_{l,j}\in SO(n)$ with
\begin{align*}
\nabla u (x) = \sum\limits_{j=1}^{\infty} \sum\limits_{l=1}^{k}  \chi_{l,j}(x) R_{l,j} U_l.
\end{align*}
Here $U_{l}:= |Q_l|^{-1}\int\limits_{Q_{l}} \nabla Z_l(y) dy$, with $Q_l$ being the box of periodicity in condition (H1)(a). The following properties are satisfied by the characteristic functions $\chi_{l,j}$
\begin{itemize}
\item[(i)] $
\sum\limits_{j=1}^{\infty} \sum\limits_{l=1}^{k} |D \chi_{l,j}| \leq C < \infty.$
\item[(ii)] $\chi_{l}(x)=\sum\limits_{j=1}^{\infty} \chi_{l,j}(x)$ for a.e. $x\in \Omega$.
\end{itemize}
\end{itemize}
\end{thm}

We again emphasize that from a physical point of view the different ground states $Z_l$ do not necessarily denote different phases, but can in principle also represent different shifts of a common ``parent phase" (c.f. the explanations in Remark \ref{rmk:phases_shifts} after Example \ref{ex:anti_ferr}). In this case the corresponding boundaries of the domains $\Omega_l$, $l\in\{1,\dots,k\}$, should be interpreted as phase-antiphase boundaries.

\subsection{Proofs}
In order to prove Theorem \ref{thm:BV_struc_gen}, we argue similarly as in Sections \ref{sec:spin} and \ref{sec:proof}.

Indeed, we first observe that (H1)-(H2) suffice to prove an analogue of the spin argument from Proposition \ref{lem:spin_arg} from Section \ref{sec:spin}:

\begin{prop}
\label{lem:spin_arg_gen}
Let $n\geq 1$ and suppose that (H1)-(H2) and \eqref{eq:en_dis} are valid. Assume that for any $m\in \N$ we have a sequence of deformation $\{(X_{m})_i\}_{i\in \Z^n}$ satisfying
\eqref{eq:en_dis} with a constant $C>0$ which is independent of $m$. Let further for $l\in\{1,\dots,k\}$
\begin{align*}
\Omega_{l,m}&:=\left\{j \in m^{-1}\Z^n \cap \Omega: \ \tau_{mj}(Q_{0}) \subset \hat{\Omega} \mbox{ and there is } R \in SO(n) \mbox{ s.t. } \right. \\
& \quad \quad \left. \dist(\nabla_d X_i,R \nabla_d Z_{l,i})_{i \in Q_{0} + mj} \leq \frac{d}{100} \right\},\\
\Omega_{b,m}&:=\left\{j \in m^{-1}\Z^n \cap \Omega:  \ \tau_{mj}(Q_{0}) \subset \hat{\Omega} \mbox{ and for all } l \in\{1,\dots,k\} \mbox{ and } \right.\\
&\left. \qquad \mbox{ for all } R \in SO(n) \mbox{ we have }  \dist(\nabla_d X_i,R \nabla_d Z_{l,i})_{i\in Q_{0} + mj} > \frac{d}{100} \right\},\\
\Omega_{bd,m}&:=\left\{j \in m^{-1}\Z^n \cap \Omega:  \ \tau_{mj}(Q_{0}) \cap (\Z^n\setminus \hat{\Omega}) \neq \emptyset \right\}.
\end{align*}
Then there exist Caccioppoli sets $\Omega_1,\dots,\Omega_k \subset \Omega$ such that
\begin{align*}
\chi_{\Omega_{l,m}} \rightarrow \chi_{\Omega_l}, \ \chi_{\Omega_{b,m}} \rightarrow
0, \ \chi_{\Omega_{bd,m}} \rightarrow
0 \mbox{ in } L^1(\Omega), \quad \Omega = \bigcup\limits_{j=1}^{k}\Omega_j.
\end{align*}
Here $\chi_{\Omega_{l,m}}, \chi_{\Omega_{b,m}}, \chi_{\Omega_l}, \chi_{\Omega_{bd,m}}:\Omega \rightarrow
\{0,1\}$ denote the characteristic functions associated with the corresponding sets.
\end{prop}

\begin{rmk}
The sets $\Omega_{b,m}$ consists of all points of comparatively large local energy density.
The set $\Omega_{bd,m}$ corresponds to a boundary layer, which is present due to the finite range interactions in our Hamiltonian and the boundedness of the set $\hat{\Omega}$.
\end{rmk}

\begin{proof}
The argument for Proposition \ref{lem:spin_arg_gen} follows from the assumptions (H1), (H2) and the energy bound \eqref{eq:en_dis} similarly as in the proof of Proposition \ref{lem:spin_arg}. 

First, we note that the energy bound \eqref{eq:en_dis} in combination with (H2) ensures that at most $\tilde{C}m^{d-1}$ lattice points have a local energy larger than a fixed constant. This can be seen by a counting argument as in Lemma \ref{lem:count}. In particular, this directly implies the vanishing of $\Omega_{b,m}$ in the limit $m\rightarrow \infty$. 

Next, we observe that (H1) combined with (H2) controls the length of the interfaces between the different phases $Z_{l,i}$, $l\in\{1,\dots,k\}$, and hence replaces Lemma \ref{lem:incomp}. More precisely, (H1) and (H2) yield that for $c_0= d/100$ the following property is satisfied:
If for some $j_0\in m^{-1}\Z^n \cap \Omega$ and some $i_0 \in \{1,\dots,n\}$ it holds that for some $R \in SO(n)$, for all $\tilde{R}\in SO(n)$ and some $l\in \{1,\dots,k\}$
\begin{align*}
\dist(R \nabla_d Z_{l,i}, \nabla_d X_i)_{i\in Q_{0} + m j_0} \leq c_0/100,\\
\dist(\tilde{R} \nabla_d Z_{l,i},\nabla_d X_i)_{i\in Q_{0} + m j_0 + e_{i_0}} > c_0/100,
\end{align*}
then for any $Q \in SO(n)$ and any $r \in \{1,\dots,k\}$
\begin{align*}
\dist(Q \nabla_d Z_{r,i}, \nabla_d X_i)_{i\in \tilde{\Lambda} + m j_0} > c_0/100.
\end{align*}
Here $\tilde{\Lambda}$ denotes the box from the condition (H2) and $e_{i_0}$ is a canonical unit vector in $\Z^n$.
This follows from an application of the triangle inequality in conjunction with the control \eqref{eq:dist_H}: As by assumption $\dist(\tilde{R} \nabla_d Z_{l,i}, \nabla_d X_i)_{i\in Q_{0} + m j_0 + e_{i_0}} > c_0/100$ for all $\tilde{R} \in SO(n)$, it suffices to show that for all $Q \in SO(n)$ and all $r\in \{1,\dots,k\}\setminus\{l\}$
\begin{align*}
\dist(Q \nabla_d Z_{r,i}, \nabla_d X_i)_{i\in \tilde{\Lambda}+ m j_0}>c_0/100.
\end{align*}
But this is a consequence of the following lower bound, which uses \eqref{eq:dist_H}: For all $Q,R \in SO(n)$
\begin{align*}
\dist(Q\nabla_d Z_{r,i}, \nabla_d X_i)_{i \in \tilde{\Lambda}+ m j_0}
&\geq \dist(Q \nabla_d Z_{r,i}, R \nabla_d Z_{l,i})_{i\in Q_{0}+ m j_0}\\
& \quad \quad  - \dist(R\nabla_d Z_{l,i}, \nabla_d X_i)_{i\in Q_{0} + m j_0}\\
& \geq d - \frac{c_0}{100}> \frac{c_0}{100}.
\end{align*}

Finally, as an additional point in the present more general set-up, we note that the size of the boundary layer $\Omega_{bd,m}$ is controlled by
\begin{align*}
|\Omega_{bd,m}| \leq C q m^{-1}, \ \Per(\Omega_{bd,m}) \leq C.
\end{align*}
This follows from the finite interaction range which is determined by the choice of $\Lambda \in \Z^n$.

As in the proof of Proposition \ref{lem:spin_arg} these three observations imply the uniform perimeter bounds on $\chi_{\Omega_{l,m}}$. These then lead to the remaining claims by general compactness arguments in the space BV (see the proof of Proposition \ref{lem:incomp} for the details).
\end{proof}

With the spin result at hand, we can proceed to the proof of Theorem \ref{thm:BV_struc_gen}. This follows along the same lines as in Section \ref{sec:proof} and mainly relies on a reduction to a one-well problem. As an additional technical aspect with respect to the proof of Theorem \ref{prop:BV_energy}, we have to deal with the fact that the ground states are now periodic in general and that only the \emph{average} deformation gradients $U_l$ are invertible. Hence, we introduce an additional auxiliary averaging step. 

\begin{proof}[Proof of Theorem \ref{thm:BV_struc_gen}]
Recalling that $u_m$ denotes the piecewise affine interpolation of $(X_{m})_i$ (c.f. \eqref{eq:interpol}), we first observe that by a truncation argument, we may again assume that $\nabla u_m$ is bounded. 
In order to avoid boundary effects, we define for each small parameter $\epsilon >0$ the set $\Omega_{\epsilon}:=\{x\in \Omega: \ \dist(x,\partial \Omega)>\epsilon\}$ and the associated Caccioppoli sets $\Omega_{l,\epsilon}:=\Omega_{\epsilon}\cap \Omega_{l}$. Here $\Omega_l$ with $l\in \{1,\dots,k\}$ are the sets from Proposition \ref{lem:spin_arg_gen}.

We note that the assumption (H2) yields that for any deformation $Y_i: \Z^n \rightarrow \R^n$
\begin{align*}
\sum\limits_{j\in \Z^n, \ \tau_j(\Lambda) \subset \tilde{\Lambda} + j_0} h(Y_{j+i})_{i\in \Lambda}
 \geq c \min\limits_{p\in \{1,\dots,k\}}\dist^p(R_{j_0} \nabla_d Z_{p,i}, \nabla_d Y_{i})_{i \in \tilde{\Lambda} +j_0},
\end{align*}
where $R_{j_0} \in SO(n)$ is the rotation which satisfies
\begin{align*}
R_{j_0}:= \argmin\limits_{R \in SO(n)} \min\limits_{p\in \{1,\dots,k\}} \dist(R \nabla_d Z_{p,i}, \nabla_d Y_{i})_{i\in \tilde{\Lambda}+j_0}.
\end{align*}
For convenience of notation and without loss of generality we assume that 
\begin{align*}
\min\limits_{p\in \{1,\dots,k\}}\dist^p(R_{j_0} \nabla_d Z_{p,i}, \nabla_d Y_{i})_{i \in \tilde{\Lambda} +j_0}
= \dist^p(R_{j_0} \nabla_d Z_{l,i}, \nabla_d Y_{i})_{i \in \tilde{\Lambda} +j_0}.
\end{align*}
Since, 
\begin{align*}
\dist(R_{j_0} \nabla_d Z_{l,i}, \nabla_d Y_i)_{i \in \tilde{\Lambda}+j_0} \geq \dist(R_{j_0} U_{l}, \nabla_d \overline{Y}_{j_0}),
\end{align*}
(where we used that $Q_{l} \subset \tilde{\Lambda}$ for all $l\in\{1,\dots,k\}$),
we then also infer
\begin{align}
\label{eq:local_bd}
\sum\limits_{j\in \Z^n, \ \tau_j(\Lambda) \subset \tilde{\Lambda} + j_0} h(Y_{j+i})_{i\in \Lambda}
 \geq c \dist^p(R_{j_0}U_l, \nabla_d \overline{Y}_{j_0}) \chi_{m\Omega_{l,\epsilon}}(j_0).
\end{align}
Here $U_l: \Z^n \rightarrow \R^n$ is the average deformation which was defined in the formulation of the theorem and $\nabla_d \overline{Y}_j: \Z^n \rightarrow \R^n$ denotes a similarly averaged field:
\begin{align*}
\nabla_d \overline{Y}_j = |Q_{l}|^{-1} \int\limits_{Q_{l}+j} \nabla_d Y(x) dx
\end{align*}
(we remark that the size of the averaging domain depends on the value of $l\in\{1,\dots,k\}$).

We apply this observation to $Y_i = (X_m)_{i}$, invoke the energy bound \eqref{eq:en_dis}, sum over all values $i\in \Z^n \cap \hat{\Omega}_{l,\epsilon}$ and use the boundedness of $\nabla_d X$ together with the bounds for $|\Omega_{bd,m}|, |\Omega_{b,m}|$. 
If, for convenience, we denote the (rescaled) piecewise affine interpolation of $\overline{X}_i$ by $\overline{u}_m$ (here the interpolation is considered in the sense of \eqref{eq:interpol} and $\overline{X}_i$ is a suitable antiderivative of $\nabla \overline{X}_i$), this leads to
\begin{align}
\label{eq:energy_a}
\int\limits_{\Omega_{l,\epsilon}} \dist^p(\nabla \overline{u}_m, SO(n)U_l) dx \leq C  L_0^n m^{-1},
\end{align} 
where $L_0$ denotes the constant from condition (H1)(a).
With \eqref{eq:energy_a} at hand, we can now argue as in the proof of Theorem \ref{prop:BV_energy}:
As in the argument for Theorem \ref{prop:BV_energy}, we first note that
by setting $x= U_l^{-1}y$ (where we use the invertibility condition from the requirement (H1)(c)), the function $\nabla \overline{u}_m(x)U_l^{-1}$ can be rewritten as
\begin{align*}
\nabla \overline{u}_m(U_l^{-1}y) U_l^{-1} = \nabla \varphi_m(y),
\end{align*}
with $\varphi_m(y)= \overline{u}_m(U_l^{-1}y)$ being piecewise affine. 
Hence,
by the change of coordinates formula and by recalling \eqref{eq:energy_a},
we deduce for each $l\in \{1,\dots,k\}$
\begin{align*}
\int\limits_{\widetilde{\Omega}} \dist^p(\chi_{\widetilde{\Omega}_{l,\epsilon}} \nabla \varphi_{m}, SO(n))dx
= \int\limits_{\widetilde{\Omega}_{l,\epsilon}} \dist^p(\nabla \varphi_{m}, SO(n))dx \leq C L_0^n m^{-1}, 
\end{align*}
where $\widetilde{\Omega}:= U_l(\Omega)$ and $\widetilde{\Omega}_{l,\epsilon}:= U_{l}(\Omega_{l,\epsilon})$.

Setting 
\begin{align*}
A_{l,m,\epsilon}(y):= \chi_{\widetilde{\Omega}_{l,\epsilon}}(y) \nabla \varphi_m(y),
\end{align*}
we note that for some constant $C>0$ which is independent of $m$,
\begin{align}
\label{eq:curl_g}
|\Curl(A_{l,m,\epsilon})| \leq C|D \chi_{\widetilde{\Omega}_{l,\epsilon}}|  \mbox{ as measures}.
\end{align}
As a consequence, Proposition \ref{prop:LL_P3} is applicable and yields
$A_{l,m,\epsilon} \rightarrow A_{l,\epsilon}$ in $L^p(\Omega)$ for $m\rightarrow \infty$ with $A_{l,\epsilon} \in SO(n)$ and
\begin{align}
\label{eq:prop41}
|D A_{l,\epsilon}|(\Omega) \leq  |\Curl(A_{l,\epsilon})|(\Omega) \leq C|D \chi_{\widetilde{\Omega}_{l,\epsilon}}|(\Omega). 
\end{align}

This entails that
there exist countably many characteristic functions
$\tilde{\chi}_{l,i}$ and matrices $R_{l,i}\in SO(n)$ with
\begin{align}
\label{eq:conv_1}
\begin{split}
&A_{l,m,\epsilon}(y) \rightarrow A_{l,\epsilon}(y)=  \sum\limits_{i=1}^{\infty} \tilde{\chi}_{l,i}(y) R_{l,i} \mbox{ in } L^p(\widetilde{\Omega}_{l,\epsilon})\\
& \mbox{and }
\sum\limits_{i=1}^{\infty} |D \widetilde{\chi}_{l,i}| \leq C |D \chi_{\widetilde{\Omega}_{l,\epsilon}}|.
\end{split}
\end{align}
In particular by definition of $A_{l,m,\epsilon}$
\begin{align*}
\nabla \varphi_m(y) \rightarrow \sum\limits_{i=1}^{\infty} \tilde{\chi}_{l,i}(y) R_{l,i} \mbox{ in } L^p(\widetilde{\Omega}_{l,\epsilon}).
\end{align*}
Changing coordinates again then leads to 
\begin{align}
\label{eq:conv}
\nabla \overline{u}_m - \sum\limits_{i=1}^{\infty} \tilde{\chi}_{l,i}(U_{l}x) R_{l,i} U_{l}
 \rightarrow 0 \mbox{ in } L^p(\Omega_{l,\epsilon}).
\end{align}
We combine \eqref{eq:conv} with the weak convergence of $\nabla u_m$:
On the one hand, the boundedness of the energy \eqref{eq:en_dis} yields 
\begin{align}
\label{eq:weak_limit1}
\nabla u_m \rightharpoonup \nabla u \mbox{ in } L^p(\Omega_{l,\epsilon}).
\end{align}
On the other hand, the periodicity of $\nabla Z_{l}$ implies that the weak limits of $\nabla u_m$ and $\nabla \overline{u}_m$ agree. Indeed, since in the phase $l$ the function $\nabla \overline{u}_m$ is defined by averaging $\nabla u_m$ over the (shifted) period cell $Q_l$, we have $\nabla \overline{u}_m = \nabla u_m \ast \chi_{Q_l,m}$, where $\chi_{Q_l,m}$ is the characteristic function of $m^{-1}Q_l \subset \Omega$. Thus, for all $\psi \in C_c^{\infty}(\Omega)$
\begin{align*}
\lim\limits_{m\rightarrow \infty} \int\limits_{\R^n} \nabla \overline{u}_m(x) \psi(x) dx
= \lim\limits_{m \rightarrow \infty} \int\limits_{\R^n} \nabla u_m(x) (\psi \ast \chi_{Q_l,m})(x) dx
=  \int\limits_{\R^n} \nabla u(x) \psi(x) dx.
\end{align*}
Here we used that $\nabla u_m \rightharpoonup \nabla u$ and $\psi \ast \chi_{Q_l,m} \rightarrow \psi$ in $L^2(\Omega)$. 
Hence, by \eqref{eq:conv} and defining $\chi_{l,i}(x) := \tilde{\chi}_{l,i}U_{l})$
\begin{align}
\label{eq:ident}
\nabla u (x) = \sum\limits_{i=1}^{\infty} \tilde{\chi}_{l,i}(x) R_{l,i} U_{l} \mbox{ for } x \in \Omega_{l,\epsilon}.
\end{align}
Using that an identity of this type holds for any $\epsilon>0$, then shows (by considering a countable family of $\epsilon \rightarrow 0$) that an identity of the form \eqref{eq:ident} holds for any subdomain of $\Omega_l$. Moreover, considering suitable diagonal sequences implies that the corresponding limits have to agree on the intersection of the various sets $\Omega_{l,\epsilon}$, in particular (up to boundary effects) neither the characteristic functions $\tilde{\chi}_{l,i}$ nor the rotations $R_{l,i}$ depend on $\epsilon>0$. Since $\nabla u$  is defined in any compact set of $\Omega$, this provides the desired representation result. The control of the BV norms of the characteristic functions $\widetilde{\chi}_{l,i}$ follows from this and from the estimate in \eqref{eq:conv_1}.
\end{proof}

\bibliographystyle{alpha}
\bibliography{citations}

\end{document}